\documentclass[11pt,letterpaper]{article}
\usepackage[T1]{fontenc}
\usepackage[letterpaper,margin=1in]{geometry}
\usepackage{lmodern,amsmath,amssymb,amsthm,mathtools,booktabs,microtype}
\usepackage[colorlinks=true,allcolors=blue]{hyperref}
\hypersetup{pdftitle={Uniform adjacency-degree probes from differential rigidity},pdfauthor={Zhipeng Lu}}
\newtheorem{theorem}{Theorem}[section]
\newtheorem{lemma}[theorem]{Lemma}
\newtheorem{proposition}[theorem]{Proposition}
\newtheorem{corollary}[theorem]{Corollary}
\theoremstyle{definition}
\theoremstyle{remark}\newtheorem{remark}[theorem]{Remark}
\DeclareMathOperator{\spanop}{span}\DeclareMathOperator{\im}{im}
\DeclareMathOperator{\tr}{tr}
\DeclareMathOperator{\diam}{diam}\DeclareMathOperator{\Spec}{Spec}
\DeclareMathOperator{\adj}{adj}
\newcommand{\R}{\mathbb R}\newcommand{\Z}{\mathbb Z}
\newcommand{\one}{\mathbf1}\newcommand{\M}{\mathcal M}
\title{Uniform adjacency--degree probes from differential rigidity}
\author{Zhipeng Lu\\[0.4em]
Shenzhen MSU--BIT University and\\
Guangdong Laboratory of Machine Perception and Intelligent Computing\\
Shenzhen 518172, Guangdong, China\\
\texttt{zhipeng.lu@hotmail.com}}
\date{}
\begin{document}
\hypersetup{pageanchor=false}
\begin{titlepage}
\maketitle
\begin{abstract}
\normalsize
We construct a fixed triangular family of noncommutative polynomials that spans the joint cyclic module of a real symmetric pair from any prescribed vector. In dimension at most $N$, when the second matrix has spectral diameter at most $N-1$, the family has $N(N+1)/2$ members, degree at most $N-1$, and a shared arithmetic circuit of size $O(N^2)$. No rank-dependent choice of coefficients is made. The proof uses a rational differential connection with a small self-adjoint residue: a primitive polynomial exterior vector for an invariant subspace must have an integer degree smaller than one.

For connected trees, the probes and the pointed adjacency--degree reconstruction theorem give one explicit self-adjoint polynomial $p_N\in\Z\langle a,d\rangle$ whose trace separates all trees of order at most $N$. Its circuit has $O(N^2)$ gates, formal degree at most $3N-2$, and polynomial bit complexity. Only the original adjacency and degree matrices, the identity, and ordinary matrix arithmetic are used. A separate integer sum-of-squares construction gives a fixed positive semidefinite representative on which the all-one vector is cyclic; its degree is $O(N^2)$. A linear trace-degree lower bound shows that the first construction has the correct asymptotic degree scale.
\end{abstract}
\thispagestyle{empty}
\end{titlepage}
\setcounter{page}{1}
\hypersetup{pageanchor=true}

\section{Introduction}
Finite-dimensionality tells us that a cyclic module has a small basis. It does not tell
us which vectors to choose before its generators are given. For matrices $H,K$ and a
vector $v$, consider the span of all word responses $w(H,K)v$. If its dimension is
$r>0$, words of length at most $r-1$ already span it. Listing them gives exponentially
many candidates; selecting a basis after evaluation makes the choice depend on the
input. Can a small family of fixed polynomial expressions replace both the full word
list and the input-dependent selection?

For real symmetric pairs, we give such a family under a bound on the spectral diameter
of $K$. In dimension at most $N$, with $\lambda_{\max}(K)-\lambda_{\min}(K)\le N-1$, an
integer triangular recurrence produces $N(N+1)/2$ polynomials of degree at most $N-1$,
computed by a shared $O(N^2)$ arithmetic circuit. Their responses span the joint cyclic
module from every prescribed vector, including when the module is proper. The first
$r$ rows already suffice when that module has dimension $r$. The degree $N-1$ is optimal
in the worst case: replacing basis selection by a fixed frame requires no increase in degree.

The key observation is that a suitable rational differential subspace cannot move. Set
$B=K/N^2$ and introduce the connection $\nabla=\partial_x+B/x+H$. For a nonzero seed
$v$, the rational span of $v,\nabla v,\nabla^2v,\ldots$ stabilizes after at most $N$
vectors and is invariant under $\nabla$. A primitive polynomial exterior vector
represents this subspace. If the subspace were nonconstant, that vector would have
positive integer degree. Self-adjointness bounds this degree by at most $N$ times the spectral diameter of $B$,
which is strictly smaller than one after scaling $K$. The subspace is
therefore constant. Invariance under the connection now forces invariance under both
matrices, so its constant part is the full joint cyclic module.

The algorithm follows by taking coefficients. Multiplying the $i$th iterate by
$N^{2i}x^i$ makes it a polynomial vector. The coefficients of the first $N$ such vectors
form our fixed frame. Differentiation supplies the lower-degree terms that distinguish
multiplication orders, but the resulting circuit uses only ordinary integer-coefficient
matrix arithmetic.

A fixed frame also gives polynomial coordinates for pointed self-adjoint
representations. Its Gram table and the two tables obtained by inserting $H$ or $K$
determine every word moment, and hence the representation on the cyclic module up to a
seed-preserving isometry. The same coordinates remain valid as the module's dimension
changes, without selecting a new basis.

Input-dependent basis exploration appears in equivalence algorithms for probabilistic
automata \cite{Tzeng}, and fixed word families span full matrix algebras on generic inputs
\cite{KlepSpenko,JiaCapel}. Here a single polynomial frame gives complete pointed
coordinates for every seed throughout the stated self-adjoint class. The degree-integrality
mechanism has precedents in small-exponent differential systems \cite{GV,BG};
Section~8 gives the detailed comparisons.

Trees give a concrete application. Adjacency spectra alone are insufficient: almost all
trees have a nonisomorphic cospectral mate \cite{Schwenk}. McKay showed that the full
adjacency--degree representation determines a tree \cite{McKay}, and the pointed
reconstruction theorem shows that the moments $\one^T w(A_T,D_T)\one$ suffice
\cite{LuLi}. Our frame replaces these moments by three finite tables. McKay's distance
polynomials convert the entries into traces; a factored balanced-radix encoding then
combines them into one self-adjoint polynomial $p_N\in\Z\langle a,d\rangle$. Its trace
separates all connected trees of order at most $N$, with $O(N^2)$ gates, degree at most
$3N-2$, and polynomial bit complexity. Every gate acts on the original matrix space, and
only the final trace is read. The factorization preserves the quadratic gate bound even
though the tables have $O(N^4)$ entries.

The question here is the cost of an algebraic representation, rather than unrestricted
tree canonization, which already has logarithmic-space algorithms \cite{Lindell,ADFK}.
Explicit pairs of nonisomorphic spiders show that any separating trace polynomial needs
degree $\Omega(N)$, matching the order of our upper bound. A second construction turns
the frame into an integer sum-of-squares matrix polynomial of degree $O(N^2)$, with
polynomial size and bit complexity. On each tree, its image is the principal module, and
powers of this single positive semidefinite operator applied to $\one$ generate that
module.

AI assistance with supporting theory, proof checking, the literature review, drafting
and editing, and Lean formalization is described in the methodology statement.

\section{The fixed construction}
For real symmetric matrices $H,K\in M_n(\R)$ and $v\in\R^n$, put
\[
\M(H,K;v)=\spanop_{\R}\{w(H,K)v:w\in\{h,k\}^{*}\}.
\]
The empty word is included. We work with straight-line matrix circuits: the inputs are $I,H,K$, and the gates are addition, ordinary matrix multiplication and multiplication by prescribed scalar constants. Reversing the order of the factors compiles the formal involution $p\mapsto p^*$; it is not an additional operation on the input.

Fix $N\ge1$ and set $\kappa=N^2$. Define polynomials $U_{k,j}\in\Z\langle h,k_0\rangle$ by
\begin{equation}\label{eq:triangle}
U_{0,0}=1,\qquad
U_{k+1,j}=\bigl(k_0+\kappa(j-k)\bigr)U_{k,j}
                 +\kappa hU_{k,j-1},
\end{equation}
where an index outside $0\le j\le k$ denotes zero. We generate rows $0,\ldots,N-1$. The symbol $k_0$ denotes the second noncommuting generator, not the row index.

\begin{theorem}[Uniform spanning probes]\label{thm:probes}
Let $1\le n\le N$, let $H,K\in M_n(\R)$ be symmetric, and suppose
\[
\lambda_{\max}(K)-\lambda_{\min}(K)\le N-1.
\]
Then, for every $v\in\R^n$,
\begin{equation}\label{eq:span}
\spanop\{U_{k,j}(H,K)v:0\le j\le k<N\}=\M(H,K;v).
\end{equation}
The family in \eqref{eq:triangle} is constructed from $N$ alone. It has $N(N+1)/2$ members, shared circuit size $O(N^2)$, and formal degree at most $N-1$. Its scalar constants have $O(\log(N+1))$ bits.
\end{theorem}

\begin{corollary}[Rank-sensitive prefix]\label{cor:rankprefix}
Under the hypotheses of Theorem~\ref{thm:probes}, if $r=\dim\M(H,K;v)$, then the rows $0\le k<r$ already span $\M(H,K;v)$. For $r=0$ this is the empty span. The circuit still generates all $N$ rows and makes no rank test.
\end{corollary}

The use of polynomial circuits, rather than a small list of individual words, matters. For example, the first mixed probes are
\[
U_{2,1}=\kappa(k_0h+hk_0),
\]
\[
U_{3,1}=\kappa\bigl(k_0^2h+k_0hk_0+hk_0^2
                    -\kappa k_0h-2\kappa hk_0\bigr).
\]
Their lower-degree terms distinguish multiplication orders. There is no assertion that the polynomials have polynomially many monomials after expansion.

For a connected tree $T$, let $A_T,D_T$ be its adjacency and degree matrices. The following is the principal application.

\begin{theorem}[A uniform linear-degree tree invariant]\label{thm:tree}
For every $N\ge1$, a deterministic algorithm running in time polynomial in $N$ constructs a self-adjoint polynomial $p_N\in\Z\langle a,d\rangle$, represented by an arithmetic circuit, such that
\[
T\cong T'
\quad\Longleftrightarrow\quad
\tr p_N(A_T,D_T)=\tr p_N(A_{T'},D_{T'})
\]
for all connected trees $T,T'$ of orders at most $N$. The circuit has $O(N^2)$ gates and formal degree at most $3N-2$. Its scalar constants have $O(N\log(N+1))$ bits. All intermediate integers and the final trace have $O(N^5\log(N+1))$ bits on these inputs.

The circuit operates on matrices of the original input order. It uses no internal trace, rank query, matrix inverse, entry query, Hadamard product, or input-dependent choice. Only the trace of its final output is read.
\end{theorem}

The graph-theoretic input in Theorem~\ref{thm:tree} is the pointed adjacency--degree reconstruction theorem for trees, stated explicitly in Section~\ref{sec:reconstruction}. The new spanning argument below does not assume a tree, an orbit partition, a simple spectrum, or a cyclic Gram matrix.

\section{A rational differential subspace cannot move}
We first prove the linear-algebraic statement responsible for uniformity. The differential language is used in the proof only; the circuit \eqref{eq:triangle} performs no differentiation or differential-equation solving.

For a linear map $B$ on $\R^n$, denote its additive action on the $r$th exterior power by
\[
B^{[r]}(u_1\wedge\cdots\wedge u_r)
=\sum_{i=1}^r u_1\wedge\cdots\wedge Bu_i\wedge\cdots\wedge u_r.
\]
This is not the multiplicative exterior power of $B$. If $B$ is symmetric, then $B^{[r]}$ is symmetric and its eigenvalues are sums of $r$ eigenvalues of $B$ with distinct indices. In particular,
\begin{equation}\label{eq:diam}
\diam\Spec(B^{[r]})\le r\,\diam\Spec(B).
\end{equation}

\begin{lemma}[Small-residue rigidity]\label{lem:rigidity}
Let $B,H\in M_n(\R)$ be symmetric. Let $V\subseteq\R(x)^n$ be a subspace of dimension $r$ invariant under
\[
\nabla f=f'+(B/x+H)f.
\]
If $r\,\diam\Spec(B)<1$, then $V=W\otimes_{\R}\R(x)$ for a constant real subspace $W\subseteq\R^n$. Moreover, $W$ is invariant under both $B$ and $H$.
\end{lemma}
\begin{proof}
The cases $r=0$ and $r=n$ are immediate, although the argument also covers the latter. Choose a rational basis of $V$ and take its exterior product. Clearing denominators and dividing the coordinates by their polynomial gcd gives a nonzero polynomial vector
\[
w(x)=w_0+w_1x+\cdots+w_\ell x^\ell
       \in\bigwedge^r\R^n[x]
\]
whose coordinates have gcd one. Consequently $w_0\ne0$ and $w_\ell\ne0$.

The line spanned by $w$ is invariant under the induced connection. Thus, for some $c(x)\in\R(x)$,
\[
xw'+(B^{[r]}+xH^{[r]})w=c(x)w.
\]
The left side is polynomial. Since the coordinates of $w$ are coprime, a polynomial Bezout identity shows that $c(x)$ is polynomial. Comparing degrees gives $\deg c\le1$. Write $c(x)=\alpha+\beta x$.

The constant and highest coefficients give
\begin{equation}\label{eq:endcoeff}
B^{[r]}w_0=\alpha w_0,
\qquad H^{[r]}w_\ell=\beta w_\ell.
\end{equation}
If $\ell>0$, the coefficient of $x^\ell$ is
\[
\ell w_\ell+B^{[r]}w_\ell+H^{[r]}w_{\ell-1}
=\alpha w_\ell+\beta w_{\ell-1}.
\]
Take its inner product with $w_\ell$. The last terms cancel by symmetry and \eqref{eq:endcoeff}, so
\begin{equation}\label{eq:integergap}
\ell=\alpha-
\frac{\langle w_\ell,B^{[r]}w_\ell\rangle}{\|w_\ell\|^2}
\le\diam\Spec(B^{[r]})<1.
\end{equation}
This contradicts the positive integer $\ell$. Hence $\ell=0$.

A constant nonzero decomposable exterior vector determines a constant subspace. Explicitly, $V$ is the set of rational vectors $u$ satisfying $u\wedge w_0=0$, and these are the rational extension of a real subspace $W$. For $u\in W$, invariance gives $(B/x+H)u\in W\otimes\R(x)$. Passing to $\R^n/W$ and equating the coefficients of $1,x$ shows $Bu,Hu\in W$.
\end{proof}

\begin{remark}
The strict inequality is meaningful. With $B=\operatorname{diag}(0,-1)$ and $H=0$, the nonconstant rational line spanned by $(1,x)^{\mathsf T}$ is horizontal for $\partial_x+B/x$. Its primitive polynomial degree is one, equal to the spectral diameter. The proof is a small-exponent argument in the elementary language of polynomial exterior vectors; related constant-reduction phenomena occur in the differential-equation literature \cite{vdPS,GV}. No exterior-power matrix is formed by our algorithm.
\end{remark}

\section{From the connection to the triangular recurrence}
\begin{proof}[Proof of Theorem~\ref{thm:probes}]
If $v=0$, both sides of \eqref{eq:span} vanish. Otherwise put $B=K/\kappa$ and
\[
q_0=v,\qquad q_{i+1}=\nabla q_i,
\qquad \nabla=\partial_x+B/x+H.
\]
Let $V$ be the $\R(x)$-span of all $q_i$. A differential Krylov chain stabilizes as soon as its next vector is dependent: differentiating a rational linear relation introduces only derivatives of its rational coefficients and the next Krylov vectors. Thus $V$ is invariant under $\nabla$, and it is spanned by $q_0,\ldots,q_{n-1}$.

Every coefficient of every $q_i$ lies in the constant joint cyclic module $\M=\M(H,K;v)$. Thus $V\subseteq\M\otimes\R(x)$. Conversely, if $r=\dim_{\R(x)}V$, then
\[
r\,\diam\Spec(B)\le\frac{N(N-1)}{N^2}<1.
\]
Lemma~\ref{lem:rigidity} makes $V$ a constant subspace invariant under $H,K$ and containing $v$. Hence
\begin{equation}\label{eq:diffspan}
V=\M\otimes\R(x).
\end{equation}

Set $f_i=x^i q_i$. The recurrence becomes
\[
f_{i+1}=xf_i'+(B-iI+xH)f_i.
\]
Induction shows that $f_i$ is polynomial of degree at most $i$. Writing
\[
f_i(x)=\kappa^{-i}\sum_{j=0}^i U_{i,j}(H,K)v\,x^j
\]
gives exactly \eqref{eq:triangle}. Let $S$ be the real span of the coefficients of $f_0,\ldots,f_{N-1}$. Then $S\subseteq\M$, while \eqref{eq:diffspan} gives $\M\otimes\R(x)\subseteq S\otimes\R(x)$. Therefore $S=\M$.

The same argument proves Corollary~\ref{cor:rankprefix}: the first dependent differential Krylov vector occurs at the dimension of $V$, which by \eqref{eq:diffspan} equals $r$. Thus $q_0,\ldots,q_{r-1}$ suffice, and so do the coefficients of $f_0,\ldots,f_{r-1}$.

Each row of \eqref{eq:triangle} has one more entry than the preceding row and uses a constant number of gates per entry. Each step increases formal degree by at most one. The stated size, degree and scalar bounds follow.
\end{proof}

\begin{remark}[What was made explicit]
A generic linear combination of a known spanning family is not being chosen. The entire family is prescribed by \eqref{eq:triangle}, and its spanning property follows before any evaluation. In particular there is no hidden hitting-set assumption and no need to certify a particular parameter sequence by inspecting the input.
\end{remark}

\begin{proposition}[The probe degree is necessary]\label{prop:probedegree}
For every $N\ge1$, there is a symmetric pair satisfying Theorem~\ref{thm:probes} for which every polynomial family spanning the cyclic module has a member of degree at least $N-1$.
\end{proposition}
\begin{proof}
Take $K=0$, $H=\operatorname{diag}(\lambda_1,\ldots,\lambda_N)$ with distinct real $\lambda_i$, and $v=\one$. The Vandermonde matrix shows that $v,Hv,\ldots,H^{N-1}v$ are independent. Every polynomial of total degree at most $d$ evaluated at $(H,0)$ sends $v$ into $\spanop\{v,Hv,\ldots,H^dv\}$, which has dimension at most $d+1$. Therefore $d\ge N-1$ is necessary to span this $N$-dimensional module. This is a degree bound; it makes no optimality claim about the number of probes or gates.
\end{proof}

\section{Pointed reconstruction and a single trace}\label{sec:reconstruction}
We use the following graph-theoretic theorem, in its pointed formulation.

\begin{theorem}[Principal-moment reconstruction; {\cite[Theorem 5.4]{LuLi}}]\label{thm:input}
If two trees $T,T'$ satisfy
\[
\one^{\mathsf T}w(A_T,D_T)\one
=\one^{\mathsf T}w(A_{T'},D_{T'})\one
\quad\text{for every word }w,
\]
then $T\cong T'$.
\end{theorem}

This input is stronger than merely saying that orbit indicators belong to the principal module. The identification of abstract pointed representations must not be replaced by an unjustified assertion that an orthogonal map is a permutation. Theorem~\ref{thm:input} supplies exactly that graph-reconstruction step. McKay's constructive theorem \cite{McKay} concerns the full adjacency--degree representation; the principal-module statement and its proof are explicitly given in \cite{LuLi}. For completeness, Section~\ref{sec:fulltrace} gives a polynomial-resource fallback using only the full-representation version of McKay.

\begin{lemma}[An overcomplete Gram table]\label{lem:gram}
Suppose $z_i=U_i(H,K)v$ span $\M(H,K;v)$ and $U_0=1$. The three tables
\[
G_{ij}=z_i^{\mathsf T}z_j,\qquad
G^H_{ij}=z_i^{\mathsf T}Hz_j,\qquad
G^K_{ij}=z_i^{\mathsf T}Kz_j
\]
determine the pointed simultaneous representation on $\M$ up to an isometry fixing its seed.
\end{lemma}
\begin{proof}
If two tables agree, send $\sum_i c_i z_i$ to $\sum_i c_i z_i'$. Equality of $G$ makes this map well defined and isometric, and the spanning hypothesis makes it onto. The other two tables imply that it intertwines $H$ and $K$: pair each desired equality with the spanning vectors $z_i'$. Since $z_0=v$, the seed is preserved.
\end{proof}

\begin{corollary}[Finite invariants for pointed representations]\label{cor:pointed}
For real symmetric pairs of dimension at most $N$ satisfying the spectral-diameter hypothesis, the three tables formed from the fixed probes in Theorem~\ref{thm:probes} are complete invariants of the pointed cyclic representation. Their entries are polynomial functions of $(H,K,v)$, of degree at most $2N-1$ in the matrix entries and degree two in the seed. Equality is equivalent to a seed-preserving isometry intertwining both generators on the cyclic modules.
\end{corollary}
\begin{proof}
The forward implication is Lemma~\ref{lem:gram}. An intertwining isometry preserves every probe vector and hence all three tables. The degree bound follows from $\deg U_i\le N-1$. Integer radix packing below uses the additional integrality and height bounds of the tree inputs.
\end{proof}

For trees, define the standard distance polynomials
\begin{equation}\label{eq:J}
\begin{aligned}
P_0&=1,\qquad P_1=a,\qquad P_2=a^2-d,\\
P_k&=aP_{k-1}-(d-1)P_{k-2}\quad(k\ge3),\qquad
J_N=\sum_{k=0}^{N-1}P_k.
\end{aligned}
\end{equation}
On a tree $T$, let $F_k(u,v)$ be the indicator of $\operatorname{dist}(u,v)=k$. Then $F_0=I$ and $F_1=A_T$. For $k=2$, the diagonal of $A_T^2$ is $D_T$, and two distinct vertices have one common neighbour exactly when their distance is two, so $F_2=A_T^2-D_T$. For $k\ge3$, count neighbours $w$ of $u$ with $\operatorname{dist}(w,v)=k-1$. The unique path from $u$ to $v$ gives one such neighbour when $\operatorname{dist}(u,v)=k$, and $\deg(u)-1$ when $\operatorname{dist}(u,v)=k-2$; otherwise the count is zero. Thus $A_TF_{k-1}=F_k+(D_T-I)F_{k-2}$. Induction proves $P_k(A_T,D_T)=F_k$. Every distance is at most $n-1\le N-1$, hence
\begin{equation}\label{eq:Jval}
J_N(A_T,D_T)=\one\one^{\mathsf T}.
\end{equation}
These are classical identities from McKay's argument \cite[Lemma 5.1 and Corollary 5.2]{McKay}. The expression $J_N$ has degree at most $N-1$. The construction retains $J_N^*$ throughout and needs no additional self-adjointness premise for this expression.

\subsection{All coefficients and all digits}
Order the $M=N(N+1)/2$ polynomials $U_i$ row by row, with $U_0=1$. List
\[
X_\nu=U_i^*SU_j,
\qquad S\in\{1,a,d\},\quad 0\le i,j<M,
\]
first by $S$ in the displayed order, then by $i,j$: the exact index is $\nu=sM^2+iM+j$, with $s=0,1,2$ for $S=1,a,d$. There are $Q=3M^2$ entries. Put
\begin{equation}\label{eq:radix}
L=3N^3,\qquad R=L^{N-1},\qquad H_N=N^2R^2,
\qquad b_N=2H_N+1.
\end{equation}
For $N=1$ take $p_1=1$ and return the identity register directly. There is only one nonempty tree of order at most one. For $N\ge2$ our final polynomial is
\begin{equation}\label{eq:final}
p_N=\sum_{\nu=0}^{Q-1}b_N^\nu
           \bigl(J_NX_\nu+X_\nu^*J_N^*\bigr).
\end{equation}
We compile this sum by factoring its radix weights. Put
\[
\begin{aligned}
Y_N&=\sum_{j=0}^{M-1}b_N^jU_j,&
Z_N&=\sum_{i=0}^{M-1}b_N^{iM}U_i,\\
W_N&=1+b_N^{M^2}a+b_N^{2M^2}d.
\end{aligned}
\]
Expanding in the order $S,i,j$ gives the exact polynomial identity
\begin{equation}\label{eq:factored}
p_N=J_NZ_N^*W_NY_N+Y_N^*W_NZ_NJ_N^*.
\end{equation}
Here $W_N^*=W_N$, since the three slots are $1,a,d$. Thus \eqref{eq:factored} is precisely \eqref{eq:final}, with the same digits and their same order.

Its only large scalar input is $b_N$. Starting from $B=b_NI$, $M$ ordinary multiplications construct $B^M=b_N^MI$, and another $M$ construct $(B^M)^M=b_N^{M^2}I$. Squaring the latter gives $b_N^{2M^2}I$. Horner's rule with bases $B$ and $B^M$ constructs $Y_N$ and $Z_N$ in $O(M)$ further gates. The scalar powers are computed registers, rather than additional prescribed constants. Unused registers may be discarded without increasing the size; in particular no $P_2$ register is needed when $N=1$.

\begin{proof}[Proof of Theorem~\ref{thm:tree}]
The case $N=1$ was settled by the identity branch. Assume $N\ge2$. The matrices $A_T,D_T$ are symmetric, with $\Spec(D_T)\subseteq[0,N-1]$. Thus Theorem~\ref{thm:probes} applies. Lemma~\ref{lem:gram} and Theorem~\ref{thm:input} show that the three tables encoded by the $X_\nu$ form a complete tree invariant.

Every entry of these tables is an integer. The induced row-sum norms of $A_T,D_T$ are at most $N-1$. Applied to vectors with maximum norm, \eqref{eq:triangle} gives
\[
\|U_{k,j}(A_T,D_T)\one\|_\infty\le L^k.
\]
The same bound holds for the induced matrix norms of $U_{k,j}$. Consequently, for $S=I,A_T,D_T$,
\[
|\one^{\mathsf T}U_i^*SU_j\one|\le N^2 R^2=H_N.
\]
By \eqref{eq:Jval},
\[
\tr p_N(A_T,D_T)
=2\sum_{\nu=0}^{Q-1}b_N^\nu h_\nu(T),
\qquad h_\nu(T)=\one^{\mathsf T}X_\nu(A_T,D_T)\one.
\]
Balanced radix is injective for digits in $[-H_N,H_N]$: at the least index where two digit strings differ, their nonzero difference has absolute value at most $2H_N<b_N$ and cannot be divisible by $b_N$. Hence equality of final traces is equivalent to equality of all table entries. In particular $h_0=n$ records the order.

The shared probe circuit has $O(N^2)$ gates, and \eqref{eq:J} constructs $J_N$ in $O(N)$ gates. The factored compilation above uses $O(M)=O(N^2)$ more gates. All operations are fixed before seeing $T$. Formal adjoints are implemented by reversing the multiplication gates. Scalar-power registers have formal degree zero, so $Y_N,Z_N$ have degree at most $N-1$ and $W_N$ has degree at most one. Equivalently, since
\[
\deg X_\nu\le2(N-1)+1=2N-1,
\]
\eqref{eq:final} has degree at most $3N-2$ and is formally self-adjoint. The coefficients in \eqref{eq:triangle}, \eqref{eq:J} and the Horner circuit are integers.

Finally $\log_2 b_N=O(N\log(N+1))$. The same row-norm bound $R$ holds for $U_i^*$ by the reversed-product recurrence. Every scalar-power register has exponent at most $2M^2$, and the partial Horner sums for $Y_N,Z_N$ have exponents at most $M-1$ and $M(M-1)$, respectively. On trees, $\|J_N\|_\infty\le N$; the intermediate products in \eqref{eq:J} also have polynomially bounded entries. Every new register is therefore bounded by polynomial factors in $N$ times $R^2b_N^{3M^2}$. This gives the stated $O(N^5\log(N+1))$ entry and trace bound. The same estimates show polynomial time for constructing and evaluating the circuit, including bit operations.
\end{proof}

\section{A fixed cyclic representative, without an adaptive schedule}\label{sec:cyclic}
The single-trace construction already proves Theorem~\ref{thm:tree}. A second application of the fixed frame gives an integer sum-of-squares AD expression whose image is the principal module and on which $\one$ is cyclic.

\subsection{One fixed rank condensation}
Assume $N\ge2$. Let $U_0=1,U_1,\ldots,U_{M-1}$ be the probes. With $L,R$ as in \eqref{eq:radix}, set
\[
H_0=(M-1)R\,2^{(N-2)(M-2)},\qquad
H_1=N!H_0^{N-1},\qquad \tau=H_1+2.
\]
Define $Z_0=1$ and
\begin{equation}\label{eq:folded}
Z_{j+1}=\sum_{i=1}^{M-1}(2^j\tau)^{i-1}U_i,
\qquad 0\le j\le N-2.
\end{equation}
These sums again use Horner's rule. Their degree is at most $N-1$, their shared size is polynomial, and their coefficient and evaluation lengths are polynomial. More explicitly, $\log\tau=O(N^4)$ and the evaluated vector entries have $O(N^6)$ bits.

\begin{lemma}[A basis in the first positions]\label{lem:prefix}
On every tree of order at most $N$, if $r=\dim\M_T$, then
\[
z_0=\one,\ z_1=Z_1\one,\ldots,z_{r-1}=Z_{r-1}\one
\]
are independent. All $N$ vectors $z_i$ belong to $\M_T$.
\end{lemma}
\begin{proof}
Work first with a variable $x$ in place of $\tau$, and pass to the quotient of $\M_T$ by $\R\one$. The coefficient vectors of
\[
F(x)=\sum_{i=1}^{M-1}U_i\one\,x^{i-1}
\]
span this quotient. Choose linear coordinates on the quotient, and by constant row operations put the resulting $r-1$ independent coordinate polynomials in a basis of distinct degrees $d_1,\ldots,d_{r-1}$. The leading coefficient of the determinant of their evaluations at
\[
x,2x,\ldots,2^{r-2}x
\]
is the product of their nonzero leading coefficients times
\[
\det\bigl[(2^{d_i})^j\bigr]_{1\le i\le r-1,\,0\le j\le r-2}\ne0.
\]
Independence in the quotient gives full column rank over $\R(x)$ in the original vertex coordinates. Thus some $r\times r$ coordinate minor of
$[\one,F(x),F(2x),\ldots,F(2^{r-2}x)]$ is a nonzero integer polynomial.

Each polynomial entry outside the first column has coefficient $\ell_1$ norm at most $H_0$. Every such minor has coefficient $\ell_1$ norm at most $N!H_0^{N-1}=H_1$. A nonzero integer polynomial with this bound has no real root greater than $H_1+1$: its leading coefficient has absolute value at least one and dominates the lower coefficients there. Therefore the same minor is nonzero at $x=\tau$.
\end{proof}

The map $f(x)\mapsto(f(x),f(2x),\ldots)$ is a folded-Wronskian rank condenser \cite{FS12,GK,FG}. We use this established rank-preservation mechanism after constructing the uniform polynomial frame. The integer height bound then supplies one prescribed specialization that works for every permitted input.

\subsection{Fraction-free orthogonal residuals}
For $0\le k<N$, let $K_k=[z_0\ \cdots\ z_{k-1}]$ and $G_k=K_k^{\mathsf T}K_k$. Define
\begin{equation}\label{eq:residual}
u_0=\one,\qquad
u_k=(\det G_k)z_k-K_k\adj(G_k)K_k^{\mathsf T}z_k\quad(k\ge1).
\end{equation}
For $k<r$, $u_k$ is a nonzero multiple of the orthogonal residual of $z_k$ against its predecessors. Hence $u_0,\ldots,u_{r-1}$ are mutually orthogonal and nonzero. For $k=r<N$, the matrix $G_r$ is invertible and $z_r\in\im K_r$, so $u_r=0$. For $k>r$ they also vanish: then $\det G_k=0$ and
\[
K_k\adj(G_k)=0.
\]
Indeed $G_k\adj(G_k)=0$, and $\ker(K_k^{\mathsf T}K_k)=\ker K_k$.

Set
\begin{equation}\label{eq:jacobi}
C^{\mathrm{aux}}=u_0u_0^{\mathsf T}
       +\sum_{k=0}^{N-2}(u_k+u_{k+1})(u_k+u_{k+1})^{\mathsf T}.
\end{equation}
In the orthonormal basis $u_i/\|u_i\|$, its restriction to $\M_T$ is tridiagonal with strictly positive off-diagonal entries. Powers of this matrix applied to the first basis vector successively have nonzero coordinates in positions $1,2,\ldots,r$. Thus it is cyclic from $u_0=\one$. Moreover, $u_0$ together with the adjacent sums $u_k+u_{k+1}$ spans $\M_T$. The kernel of their positive-semidefinite outer-product sum is exactly $\M_T^\perp$, so its image is $\M_T$.

\begin{theorem}[Uniform cyclic principal matrix]\label{thm:uniformcyclic}
There is a deterministic polynomial-time construction, from $N$ alone, of an integer AD circuit $C_N$, formally a sum of Hermitian squares, with degree at most $2(4N-3)(N-1)$ for $N\ge2$ and polynomial size and bit complexity such that, on every tree of order at most $N$,
\[
C_N\succeq0,\qquad \im C_N=\M_T,
\qquad \spanop\{\one,C_N\one,\ldots,C_N^{N-1}\one\}=\M_T.
\]
All its gates act on the original matrix space. Formula \eqref{eq:residual} requires no determinant, inverse, rank, or trace oracle.
\end{theorem}
\begin{proof}
For $N=1$ return $C_1=I$. Assume $N\ge2$, and abbreviate $J_N$ to $J$. Compute original-order matrix registers
\[
B_{ij}=JZ_i^*Z_jJ.
\]
On a tree of order $n$, these satisfy $B_{ij}=g_{ij}J$, where $g_{ij}=z_i^{\mathsf T}z_j$, and $J^2=nJ$. A homogeneous scalar polynomial $f$ of positive degree $q$ in the Gram entries can therefore be compiled so that
\begin{equation}\label{eq:homogeneous-lift}
\widehat f(B)=n^{q-1}f(g)J.
\end{equation}
Indeed, addition of equal degrees preserves this identity, while multiplication of positive degrees $p,q$ gives the factor $n^{p+q-1}$ by $J^2=nJ$. Degree-zero constants act by prescribed integer scaling; they are never represented by a putative multiplicative identity $J$.

Division-free determinant circuits are classical \cite{Soltys,Bird}. We use Bird's homogeneous recurrence: for a $k\times k$ scalar matrix $G$, put $F_1=G$ and $F_{q+1}=\mu(F_q)G$, where $\mu(F)$ keeps the strictly upper-triangular entries, has zero entries below the diagonal, and has diagonal entries $-\sum_{\ell>i}F_{\ell\ell}$. Then $(-1)^{k-1}(F_k)_{00}=\det G$. Each stage has homogeneous degree $q$; shared ordinary matrix multiplication implements the scalar recurrence with $O(k^4)$ arithmetic gates. Thus \eqref{eq:homogeneous-lift} applies directly, with all integer coefficients implemented as scaling gates. The auxiliary $k\times k$ arrays index original-order matrix registers. Let $b_k=(g_{ik})_{i<k}$ and let $G_k^{(i)}$ be $G_k$ with column $i$ replaced by $b_k$. The degree-$k$ determinant circuits give
\[
\widehat\Delta_k=n^{k-1}\det(G_k)J,\qquad
\widehat\Delta_{k,i}=n^{k-1}\det(G_k^{(i)})J.
\]
Define $V_0=J$ and, for $1\le k<N$,
\[
V_k=Z_k\widehat\Delta_k-\sum_{i<k}Z_i\widehat\Delta_{k,i}.
\]
The polynomial column-replacement identity
$\det(G_k^{(i)})=(\adj(G_k)b_k)_i$ holds also when $G_k$ is singular. Thus on trees $V_k=n^{k-1}u_k\one^{\mathsf T}$. With $t_0=\one$ and $t_k=n^{k-1}u_k$ for $k\ge1$, set
\begin{equation}\label{eq:integer-sos}
C_N=V_0V_0^*+\sum_{k=0}^{N-2}(V_k+V_{k+1})(V_k+V_{k+1})^*.
\end{equation}
Its value is $n[t_0t_0^{\mathsf T}+\sum_k(t_k+t_{k+1})(t_k+t_{k+1})^{\mathsf T}]$. The $t_i$ have the same orthogonal nonzero prefix and zero tail as the $u_i$. The Jacobi and image arguments following \eqref{eq:jacobi} apply with these positive rescalings, proving the three conclusions. Formula \eqref{eq:integer-sos} is formally self-adjoint and evaluates to a positive semidefinite matrix on every real symmetric pair; the image and cyclicity claims use the tree frame properties.

Since $\deg Z_i,\deg J\le N-1$, we have
\[
\deg B_{ij}\le4(N-1),\quad
\deg\widehat\Delta_k,\deg\widehat\Delta_{k,i}\le4k(N-1),
\]
and $\deg V_k\le(4k+1)(N-1)$. Equation~\eqref{eq:integer-sos} consequently has degree at most $2(4N-3)(N-1)$. There are polynomially many determinant circuits of order at most $N$. Their auxiliary arrays consist of original-order matrix registers; no determinant oracle or larger matrix input is used. The standard homogeneous determinant recurrences use only integer addition and multiplication and have polynomial intermediate coefficient height. Together with the $O(N^6)$-bit folded-frame entries and degrees at most $N$ in the Gram variables, this gives polynomial bit complexity for every register. The construction does not assert that the spectrum of $C_N$ alone separates trees.
\end{proof}

\section{A linear trace-degree lower bound}\label{sec:lower}
The bound $3N-2$ is not claimed to be optimal in its leading constant. Its order is optimal.

\begin{proposition}\label{prop:lower}
For $m\ge2$, let $T_m$ and $T'_m$ be the spiders with arm lengths
\[
(m,m,m+2)\quad\text{and}\quad(m,m+1,m+1),
\]
respectively. Both have $3m+3$ vertices and maximum degree three. Every AD word containing fewer than $2m$ occurrences of $A$ has the same trace on these two trees. Consequently, any single-trace polynomial separating all trees of order at most $N$ has degree at least $2\lfloor(N-3)/3\rfloor$ for $N\ge9$.
\end{proposition}
\begin{proof}
A trace word with $s$ adjacency factors is a sum over closed walks of length $s$, with products of vertex degrees attached at prescribed positions. On a tree, such a walk traverses every edge of its image at least twice. Its image therefore has at most $s/2<m$ edges.

Call the center and the three leaves special. The image cannot contain two special vertices. If it contains the center, its degree-labelled embeddings are the same in both spiders, since every arm has length at least $m$. If it contains a leaf but not the center, its image is a short path and the three leaf neighbourhoods agree. If it contains no special vertex, it is a path of length $\ell<m$ inside the degree-two portion of an arm. An arm of length $L_a$ has $L_a-1-\ell$ such undirected placements, so their total is $\sum_a L_a-3(\ell+1)$, equal for the two spiders. Orientations, choices of starting position, and repeated-walk patterns contribute the same multiplicities to each placement. Hence every degree-labelled walk count, and therefore every trace word in question, agrees. The case $s=0$ follows directly from the common degree multiset.

A polynomial of total degree less than $2m$ is a linear combination of such words. The unique degree-three center has two leaves at distance $m$ in $T_m$ and one in $T'_m$, so the spiders are nonisomorphic. Taking $m=\lfloor(N-3)/3\rfloor$ proves the stated bound.
\end{proof}

\section{Relation to prior constructions}
The graph-theoretic inputs are McKay's full adjacency--degree reconstruction theorem \cite{McKay} and the pointed reconstruction theorem \cite[Theorem 5.4]{LuLi}. McKay's distance-polynomial ingredient is used explicitly in the construction of $J_N$. Tree canonization in logarithmic space goes back to Lindell \cite{Lindell}; Arvind, Datta, Faris, and Khan give a polynomial canon in logarithmic space \cite{ADFK}. The computational task here prescribes an original-dimension, low-degree matrix circuit and a single final scalar readout.

Trace coordinates themselves are classical. Procesi proves trace generation for simultaneous-conjugation invariants and treats orthogonal invariants using words in the matrices and their transposes \cite{Procesi}; Razmyslov studies trace identities of full matrix algebras in characteristic zero \cite{Razmyslov}. Bounds on the length of generated matrix algebras control the span of all short words, from the work of Paz and Pappacena to Shitov's $O(n\log n)$ bound for $n\times n$ matrices \cite{Paz,Pappacena,Shitov}. Sprung gives a further log-log saving in this general bound \cite{Sprung}. The full algebra can have dimension $n^2$, whereas a cyclic module from one seed has dimension at most $n$. These word-length bounds do not specify the fixed quadratic-size family of mixed polynomials required here.

Fixed short word families are also known to span the full matrix algebra on a Zariski-open set of tuples: Klep and \v{S}penko construct sweeping words, and Jia and Capel develop the related generic quantum Wielandt bound \cite{KlepSpenko,JiaCapel}. This generic guarantee does not supply a frame valid throughout our stated self-adjoint class, including proper cyclic modules and arbitrary seeds.

Pointed moments also define the series $w\mapsto v^{\mathsf T}w(H,K)v$, a linear representation of the form used for weighted automata. In the probabilistic setting, Tzeng's polynomial-time equivalence algorithm explores a representation and prunes linearly dependent vectors \cite{Tzeng}. The present frame instead supplies polynomial coordinates before the input is evaluated, under the stated self-adjointness and spectral-diameter conditions. The same redundant coordinates apply to proper cyclic modules and across changes in their dimension.

Short separating words for matrix tuples and their algorithmic selection are studied in invariant theory \cite{DM}. Trace-product constructions, ROABPs and explicit separating sets are directly linked by Forbes and Shpilka \cite{FS}. The present proof does not invoke a polynomial-size hitting set for general ROABPs. It uses real self-adjoint matrices with a bounded spectral diameter and a scaled differential residue. Removing self-adjointness would invalidate the Rayleigh-quotient step \eqref{eq:integergap}; the result should not be advertised as a derandomization for arbitrary complex matrix tuples.

Differential modules provide a standard framework \cite{vdPS}. The integrality of subbundle degree is already central to small-exponent constant-reduction arguments: Gontsov and Vyugin use a degree bound greater than $-1$ to force degree zero \cite[Lemma 2]{GV}. Barkatou and Gontsov also formulate reduction criteria using differences of formal exponents \cite[Theorem 2.3]{BG}. These criteria address Liouvillian solvability; their irregular cases require pairwise distinct eigenvalues of the leading matrices. Lemma~\ref{lem:rigidity} concerns every rational invariant subspace of the specified self-adjoint connection and permits repeated eigenvalues of $H$. Its application yields the explicit fixed-seed spanning recurrence and the circuit consequences. The folded-Wronskian step in Section~\ref{sec:cyclic} uses the established constructions of \cite{FS12,GK,FG}, followed by the single bounded-height specialization proved here.

\appendix
\section{A full-representation route using only McKay}\label{sec:fulltrace}
This section records an independent graph-theoretic fallback if the principal-moment formulation is not used. It has larger, but still polynomial, parameters.

Equip $M_n(\R)$ with the Frobenius inner product and seed $I$. Left multiplication by $A$ and $D$ gives symmetric operators in dimension $n^2\le N^2$; the latter has spectral diameter at most $N-1$. Theorem~\ref{thm:probes} with bound $N^2$ and $\kappa=N^4$ gives $O(N^4)$ polynomials of degree at most $N^2-1$ spanning $\langle A,D\rangle$.

The three Frobenius Gram tables
\[
\tr(U_i^*U_j),\quad\tr(U_i^*AU_j),\quad\tr(U_i^*DU_j)
\]
determine the pointed left-regular representation with seed $I$. Equality therefore implies equality of $\tr w(A,D)$ for every word $w$. For every fixed matrix polynomial $q$, equality also holds for $\tr(q^k)$ for every $k$. Newton identities therefore give the same characteristic polynomial for $q$ on the two inputs. McKay's theorem then distinguishes trees.

The factoring used in \eqref{eq:factored} applies to these trace tables as well. If $M'=O(N^4)$ is the number of probes and $b$ exceeds twice a uniform bound on the integer table entries, put $Y=\sum_jb^jU_j$, $Z=\sum_ib^{iM'}U_i$, and $W=1+b^{M'^2}a+b^{2M'^2}d$. Then
\[
p_N^{\mathrm{full}}=Z^*WY+Y^*WZ
\]
packs all three tables using $O(M')=O(N^4)$ gates and degree at most $2N^2-1$. All evaluations still use matrices of the original order; the Frobenius-space enlargement is only in the spanning proof. This construction needs no $J$ expression and has polynomial bit complexity. It uses the classical full-representation tree theorem and does not give the sharper linear-degree bound.

\section{Reproducibility and scope of computation}
The accompanying verification directory contains the fixed emitter \texttt{quadratic\_ad\_circuit.py}, the independent checks \texttt{verify\_current.py} and \texttt{verify\_exports.py}, and their recorded results. The emitter sees only $N$. Its exported example \texttt{p5.quadratic.circuit.json.gz} has 407 gates and output formal degree 13. At $N=1$ it returns the identity. The allowed instructions are identity, adjacency, degree, zero, addition, multiplication and integer scaling; adjoints are compiled by reversing products. An evaluator may skip arithmetic on zero entries without changing the emitted DAG.

The file \texttt{current\_checks.json} records syntax checks for $N=1,\ldots,15$, exact evaluation on all 14 nonisomorphic nonempty trees of order at most six with the common bound $N=6$, three $N=1,2$ boundary inputs, and 42 real symmetric integer-matrix pairs. The checks compare the recurrence with an independently generated word span or rooted-tree orbit count, verify the actual distance expression and balanced-radix decoding, and record integer fingerprints. Modular rank is used as a characteristic-zero lower bound and compared with an independently justified upper bound.

The file \texttt{export\_checks.json} records equality of the Python emitter with the actual Lean instruction arrays for $N=1,\ldots,15$. It also executes the actual division-free PSD factory for $N=1,\ldots,4$ on 11 tree inputs and 12 further symmetric matrix pairs. On trees, its values are compared with an independently computed fraction-free Gram-residual formula, and image, cyclicity and positivity are checked exactly. These are finite regressions, not universal proofs that two separately implemented emitters agree.

The Lean 4/mathlib proof-audit sources, source-to-statement table, pinned dependencies, build records and reproduction scripts are provided at \url{https://github.com/Lzp88/AD}. The project integrates the rational differential spanning proof, the literal allowed-gate compiler, and the previously audited principal-moment reconstruction input. The reconstruction interface is the actual finite-tree word-moment theorem of \cite[Theorem 5.4]{LuLi}; it constructs a graph isomorphism rather than assuming that a cyclic isometry is a permutation. Its proof is reused through the public \texttt{PrincipalMomentAudit} namespace.

The source-to-statement table names the top-level declarations and their full hypotheses. In particular, it distinguishes the actual spectral-diameter spanning theorem from its quadratic-form helpers, includes arbitrary prescribed seeds and proper cyclic modules, and connects the actual tree distance identity to the compiled single-trace classification theorem, including $N=1$. The resource proofs bound every stored instruction's formal degree and every integer register's height. The PSD endpoint concerns the complete fixed Bird/Cramer/sum-of-squares factory, with no supplied rank or successful specialization premise. The spider endpoint uses actual nonisomorphic trees and arbitrary real polynomial coefficients.

The Frobenius-space spanning and three-table results in Appendix~A are also checked. Its final tree-reconstruction proof in the formal project uses the audited principal-moment input, via the actual distance expression, whereas the independent written argument in Appendix~A uses McKay and Newton identities. The source map records this difference in proof dependencies; it does not label the formal reconstruction proof as independent of the pointed input.

The frozen build records source and imported-object hashes and rejects warnings, admitted proofs, project-local axioms and transitive axiom dependencies beyond \texttt{propext}, \texttt{Classical.choice} and \texttt{Quot.sound}. The package contains the complete local import closure, toolchain and dependency pins, a rebuild script, and the source-to-statement table. The records describe the checked source snapshot; finite experiments remain separate evidence.

\section*{AI methodology}
ChatGPT and OpenAI Codex helped develop parts of the supporting theory, including the formulation and refinement of intermediate results, and assisted with checking the proofs of most theoretical results. They also assisted with the literature review, drafting and editing, and the development and verification of Lean formalizations. The author takes full responsibility for the content of this paper.

\end{document}